\documentclass[12pt, twoside, reqno]{amsart}
\usepackage{amscd}
\usepackage{bbm}
\usepackage{dsfont}
\usepackage{enumerate}
\usepackage{latexsym}
\usepackage{srcltx}
\usepackage{amsfonts}
\usepackage{amssymb}
\usepackage{datetime}
\usepackage{setspace}
\usepackage{enumerate}
\usepackage{amsthm}
\usepackage{graphicx}
\usepackage{epstopdf}

\newtheorem{theorem}{Theorem}[section]
\newtheorem{proposition}[theorem]{Proposition}
\newtheorem{lemma}[theorem]{Lemma}

\theoremstyle{definition}
\newtheorem{example}[theorem]{Example}
\newtheorem{definition}[theorem]{Definition}

\begin{document}
\title{ Spectrum of Weighted Composition Operators \\
Part X \\
The spectrum and essential spectra of weighted automorphisms of the polydisc algebra.}

\author{Arkady Kitover}

\address{Community College of Philadelphia, 1700 Spring Garden St., Philadelphia, PA, USA}

\email{akitover@ccp.edu}

\author{Mehmet Orhon}

\address{University of New Hampshire, 105 Main Street
Durham, NH 03824}

\email{mo@unh.edu}

\subjclass[2010]{Primary 47B33; Secondary 47B48, 46B60}

\date{\today}

\keywords{Weighted composition operators, spectrum, Fredholm spectrum, essential spectra}

\begin{abstract}
  We investigate the spectrum and the essential spectra of weighted automorphisms of the polydisc algebra $\mathds{A}^n$. In the case $n=2$ we provide a detailed and, in most cases, complete description of these spectra.
\end{abstract}

\maketitle

\markboth{Arkady Kitover and Mehmet Orhon}{Spectrum of weighted automorphisms.}

\section{Introduction}

The spectrum of weighted automorphisms of the disc-algebra was described by Kamowitz in~\cite{Kam}. In~\cite{Ki} the first named author obtained a full description of the spectrum of weighted automorphisms of uniform algebras. The paper~\cite{Ki1} contains a description of essential spectra of weighted automorphisms of the algebra $C(K)$. In~\cite{KO} the authors obtained some general results about the spectrum and essential spectra of weighted isometrical endomorphisms of uniform algebras and applied these results to the problem of describing the essential spectra of weighted automorphisms of the disc-algebra. This description is complete in the case of automorphisms generated by elliptic or parabolic Möbius transformations, and partial in the case of hyperbolic transformations (see~\cite[Problem 3.7]{KO}).

The goal of the current paper is to apply the results obtained in~\cite{Ki, Ki1, KO} to the problem of describing essential spectra of weighted automorphisms of the polydisc algebra $\mathds{A}^n$. While our methods can be applied in the case of any $n \in \mathds{N}$, we decided to provide as many details as possible in the case $n=2$. As the reader will see, even in this case the multitude of distinct possibilities is quite large. 

\section{Notations and preliminaries}

The following notations are used throughout the paper.

\noindent $\mathfrak{M}_A$ and $\partial A$ are the space of maximal ideals and the Shilov boundary of
a commutative Banach algebra $A$, respectively.

\noindent Let $A$ be a commutative unital Banach algebra and $U$ be an automorphism of $A$. We will often instead of $U$ write $T_\varphi$, where $\varphi$ is the homeomorphism of $\mathfrak{M}_A$ onto itself generated by $U$.

\noindent $\mathds{N}$ is the set of all positive integers.

\noindent $\mathds{Z}$ is the set of all integers.

\noindent $\mathds{C}$ is the field of all complex numbers.

\noindent We denote by $\mathds{D}$ and by $\mathds{U}$ the closed and the open unit disk, respectively.

\noindent $\mathds{T}$ is the unit circle.

\noindent $\mathds{AN}(r,R)$, where $0 < r < R < \infty$, is the annulus $\{\lambda \in \mathds{C}: r \leq |\lambda| \leq R$\}.

\noindent For $n=2,3, \ldots$, $\mathds{D}^n$, $\mathds{U}^n$, and $\mathds{T}^n$ denote the closed unit polydisc, the open unit polydisc, and the $n$-dimensional unit torus, respectively. 

\noindent All linear spaces are considered over the field $\mathds{C}$.

\noindent $C(K)$ is the space of all complex-valued continuous functions on a Hausdorff compact space $K$ endowed with the standard norm.

\noindent $\mathds{A}^n, n = 2,3, \ldots$ is the polydisc algebra: the algebra of all functions analytic in $\mathds{U}^n$ and continuous in $\mathds{D}^n$.

\noindent It is well known (see e.g~\cite{Ru}) that $\mathfrak{M}_{\mathds{A}^n} = \mathds{D}^n$ and $\partial \mathds{A}^n = \mathds{T}^n$.

\noindent Let $K$ be a compact Hausdorff space and $\varphi$ be a homeomorphism of $K$ onto itself. Then, $\varphi^0(k) = k, k \in K$, and $\varphi^n = \varphi^{n-1} \circ \varphi, n \in \mathds{Z}$.

\noindent Let $K$ be a Hausdorff compact space, $\varphi$ be a continuous map of $K$ into itself, and $w \in C(K)$. Then, $w_1 = w$ and $w_n = w_{n-1}(w \circ \varphi^{n-1}), n = 2,3, \ldots$.

\noindent Let $X$ be a Banach space and $T : X \rightarrow X$ be a continuous linear operator. 

\noindent $\sigma(T)$ is the spectrum of $T$.

\noindent $\rho(T)$ is the spectral radius of $T$. 

\begin{equation*}
 \rho_{min}(T) = \left\{
    \begin{array}{ll}
      1/\rho(T^{-1}, & \hbox{if the operator $T$ is invertible;} \\
      0, & \hbox{otherwise.}
    \end{array}
  \right.
\end{equation*}

\noindent $\sigma_{a.p.}(T)$ is the approximate point spectrum of $T$, i.e.
\begin{equation*}
   \sigma_{a.p.}(T) = \{\lambda \in \mathds{C} : \exists x_n \in X, \|x_n\|=1, Tx_n - \lambda x_n \mathop \rightarrow_{n \rightarrow \infty} 0\}.
\end{equation*}
According to this definition the point spectrum of $T$ is a subset of $\sigma_{a.p.}(T)$.

\noindent $\sigma_r(T) = \sigma(T) \setminus \sigma_{a.p.}(T)$ is the residual spectrum of $T$.

\noindent Recall that a continuous linear operator $T$ on a Banach space $X$
 is called upper semi-Fredholm (respectively, lower semi-Fredholm) if its image $R(T)$ is closed in  $X$ and $\dim \ker T < \infty$ (respectively, if $codim R(T) < \infty$).  

\noindent The operator $T$ is called semi-Fredholm if it is either upper semi-Fredholm or lower semi-Fredholm, and it is called Fredholm if it is both upper semi-Fredholm and lower semi-Fredholm.

\noindent The operator $T$ is called Weil operator if $\dim \ker T = codim R(T) < \infty$.

\noindent $\sigma_{sf}(T)$ is the semi-Fredholm spectrum of $T$, i.e. $\sigma_{sf}(T) = \{\lambda \in \mathds{C}: \lambda I - T$ is not semi-Fredholm$\}$.

\noindent $\sigma_{usf}(T)$ is the upper semi-Fredholm spectrum of $T$, i.e. $\sigma_{usf}(T) = \{\lambda \in \mathds{C}: \lambda I - T$ is not an upper semi-Fredholm operator$\}$. It is well known (see e.g.~\cite{EE}) that
\begin{equation*}
  \begin{split}
  & \sigma_{usf}(T) = \{\lambda \in \mathds{C} : \exists x_n \in X, \|x_n\|=1, Tx_n - \lambda x_n \mathop \rightarrow_{n \rightarrow \infty} 0 ,\\
  & \text{and the sequence} \; x_n \; \text{is singular, i.e. it does not contain} \\
   & \text{a norm convergent subsequence}\}.
  \end{split}
\end{equation*}

\noindent $\sigma_{lsf}(T)$ is the lower semi-Fredholm spectrum of $T$, i.e. $\sigma_{lsf}(T) = \{\lambda \in \mathds{C}: \lambda I - T$ is not a lower semi-Fredholm operator$\}$. 
\begin{equation*}
  \begin{split}
  & \sigma_{lsf}(T) = \{\lambda \in \mathds{C} : \exists x_n^\prime \in X^\prime, \|x_n^\prime\|=1, T^\prime x_n^\prime - \lambda x_n^\prime \mathop \rightarrow_{n \rightarrow \infty} 0 ,\\
  & \text{and the sequence} \; x_n^\prime \; \text{is singular}.
    \end{split}
\end{equation*}

\noindent $\sigma_f(T)$ is the Fredholm spectrum of $T$, i.e. $\sigma_f(T) = \{\lambda \in \mathds{C}: \lambda I - T$ is not a Fredholm operator$\}$.  

\noindent $\sigma_w(T)$ is the Weil spectrum of $T$, i.e. $\sigma_f(T) = \{\lambda \in \mathds{C}: \lambda I - T$ is not a Weil operator$\}$. 

\noindent It is well known (see e.g~\cite{EE}) that
\begin{itemize}
  \item $\sigma_{sf}(T) \subseteq \sigma_{usf}(T) \subseteq \sigma_f(T) \subseteq \sigma_w(T)$,
  \item $\sigma_{usf}(T) = \sigma_{lsf}(T^\prime)$, $\sigma_{lsf}(T) = \sigma_{usf}(T^\prime)$.
\end{itemize}

\noindent In the case when operator $T$ is defined and bounded on distinct Banach spaces, $\sigma(T, X)$ means the spectrum of $T$ on the Banach space $X$. The same agreement relates to all the spectra considered above.

\section{Axillary results} 

Our investigation of essential spectra of weighted automorphisms of the algebra $\mathds{A}^2$ is based on several lemmas stated below.

\begin{lemma} \label{l1} (see~\cite[Proposition 2.8]{KO}
  Let $A$ be a unital uniform algebra. Assume that $\partial A$ has no isolated points. Let $T_\varphi$ be a periodic automorphism of $A$, i.e. there is an $n \in \mathds{N}$ such that $T_\varphi^n = I$. Let $w \in A$ and $T = wT_\varphi$. For any $t \in \mathfrak{M}_A$ let $p(t)$ be the smallest positive integer such that $\varphi^{p(t)}(t) = t$. Then,
\begin{enumerate}
  \item $\sigma(T) = \{\lambda \in \mathds{C}: \exists t \in \mathfrak{M}_A, \lambda^{p(t)} = w_{p(t)}(t)\}$.
  \item $\sigma_{a.p.}(T) = \sigma_{sf}(t) =\{\lambda \in \mathds{C}: \exists t \in \partial A, \lambda^{p(t)} = w_{p(t)}(t)\}$
\end{enumerate}
\end{lemma}

Let us recall the following definition.

\begin{definition} \label{d1}
  A uniform algebra $A$ is called analytic if for any nonempty open subset $U$ of $\partial A$ and for any $a \in A$ we have $a|U \equiv 0 \Rightarrow a=0$.
\end{definition}

\begin{lemma} \label{l2} (see~\cite[Theorem 2.9]{KO})
    Let $A$ be a unital analytic uniform algebra, $U = T_\varphi$ be a non-periodic automorphism of $A$, and $T = wT_\varphi$. Assume that
\begin{enumerate} [(a)]
  \item $\partial A$ is a connected set.
  \item There are periodic automorphisms, $U_n = T_{\varphi_n}$ of $A$ such that for any $t \in \partial A$ we have $\varphi_n(t) \rightarrow \varphi(t)$.
\item Let $\mu \in (C(\partial A))^\prime$ be a $\varphi$-invariant probability measure. Then, $\mu$ is $\varphi_n$-invariant, $n \in \mathds{N}$.
\end{enumerate}
Then, 
\begin{enumerate}
  \item $\sigma_{a.p.}(T) = \sigma_{sf}(T)$.
  \item The sets $\sigma(T)$ and $\sigma_{sf}(T)$ are connected rotation invariant subsets of $\mathds{C}$.
  \item If either $w \in A^{-1}$ or $\min \limits_{s \in \partial A} |w(s)| = 0$, then $\sigma(T) = \sigma_{sf}(T)$.
\end{enumerate}
\end{lemma}

\begin{lemma} \label{l3} (see~\cite[Lemma 3.6, Theorem4.2, Corollary 4.4]{Ki} and~\cite[Theorem 2.1, Proposition 2.11]{KO})

  Let $A$ be an analytic uniform unital algebra. Let $T_\varphi$ be a non-periodic automorphism of $A$, $w \in A$, and $T = w T_\varphi$. Then,
\begin{enumerate}
  \item $\sigma(T)$ is a connected rotation invariant subset of $\mathds{C}$.
  \item $\sigma_{a.p.}(T)$ is rotation invariant and $\sigma_{a.p.}(T) = \sigma_{a.p.}(T, C(\partial A))$.
  \item Let $\lambda \in \mathds{C}$, then $\lambda \in \sigma_{a.p.}(T)$ if and only if there is a $t \in \partial A$ such that $|w_n(t)| \geq |\lambda|^n$ and $|w_n(\varphi^{-n}(t)) \leq |\lambda|^n$ for any $n \in \mathds{N}$.
  \item If $\partial A$ has no isolated points, then $\sigma_{usf}(T) = \sigma_{a.p.}(T)$.
  \item If $\partial A$ has no isolated points and there is a $u \in \partial A$ such that $|w_n(u)| \leq |\lambda|^n$ and $|w_n(\varphi^{-n}(u)) \geq |\lambda|^n$ for any $n \in \mathds{N}$, then $\lambda \in \sigma_{lsf}(T)$.
\end{enumerate}
\end{lemma}

\begin{lemma} \label{l4} (see~\cite{Ki1}) Let $A$ be a unital uniform algebra and $T = wT_\varphi$ be a weighted automorphism of $A$. Assume that $\partial A$ contains no isolated points. Assume also that $\partial A$ is the union of $\varphi$ and $\varphi^{-1}$ invariant nonempty sets $K_1, K_2$ and $O$ such that the sets $K_1$ and $K_2$ are closed, $(K_1 \cup K_2) \cap O = \emptyset$,  and for any $k \in O$ we have
\begin{equation*}
  \begin{split}
 & \bigcap \limits_{m=1}^\infty  cl\{\varphi^i(k) : i \geq m\} = K_1 , \\
& \bigcap \limits_{m=1}^\infty cl\{\varphi^{-i}(k) : i \geq m\} = K_2.
\end{split}
\end{equation*}
Then, the following statements hold.
\begin{enumerate}
  \item If $0 \leq r = \rho(T, C(K_2)) < R = \rho_{min}(T, C(K_1))$, then either $\mathds{AN}(r,R) \subseteq \sigma_{usf}(T)$, or, if $r = 0$, $R\mathds{D} \subseteq \sigma_{usf}(T)$.
  \item If $0 \leq r = \rho(T, C(K_1)) < R = \rho_{min}(T, C(K_2))$, then either $\mathds{AN}(r,R) \subseteq \sigma_{lsf}(T)$, or, if $r = 0$, $R\mathds{D} \subseteq \sigma_{lsf}(T)$.
  \item If $\rho_{min}(T, C(K_1)) = \rho(T, C(K_1)) = \rho_{min}(T, C(K_2)) = \rho(T, C(K_2)) = r$, then
$r\mathds{T} \subseteq \sigma_{sf}(T)$.
\item Moreover, if there is a $k \in O$ such that $w(k) =0$ and the set $\{n \in \mathds{Z} : w(\varphi^n(k)) =0\}$ is finite, then $\rho(T, C(K_2))\mathds{D} \subseteq \sigma_{usf}(T)$ and 
$\rho(T, C(K_1))\mathds{D} \subseteq \sigma_{lsf}(T)$.
\end{enumerate}
 \end{lemma}

\begin{lemma} \label{l5} (see~\cite{Ki}) Let $A$ be a unital uniform algebra and $T = wT_\varphi$ be a weighted automorphism of $A$. Then,
\begin{equation*}
  \rho(T) = \max \limits_{\mu \in M_\varphi} \exp \int \ln{|w|} d\mu,
\end{equation*}
  where $M_\varphi = \{\mu \in C^\prime(\partial A) : T_\varphi^\prime \mu = \mu, \mu \geq 0, \|\mu\| = 1$\}.
\end{lemma}

\section{Essential spectra of weighted automorphisms of the algebra $\mathds{A}^2$}

It is well known (see~\cite{Ru}) that every automorphism of the algebra $\mathds{A}^2$ is the composition operator $T_\Phi$, where the map $\Phi$ is on of the following
\begin{enumerate} [(a)]
  \item $\Phi(z_1, z_2) = (\varphi(z_1), \psi(z_2)), (z_1, z_2) \in \mathds{D}^2$,
  \item $\Phi(z_1, z_2) = (\psi(z_2), \varphi(z_1)), (z_1, z_2) \in \mathds{D}^2$,
\end{enumerate}

\noindent where $\varphi$ and $\psi$ are Möbius transformations of $\mathds{D}$.

In the case (b) the map $\Phi^2$ is of the form (a) and taking into consideration the spectral mapping theorem for essential spectra (see~\cite[Corollary 3.61, p. 149]{Ai}) we can without loss of generality assume that the map $\Phi$ is of the form (a).

We also can and will without loss of generality consider only elliptic Möbius transformations of the form $\varphi (z) = \alpha z, z \in \mathds{D}, \alpha \in \mathds{T}$.

\begin{proposition} \label{p12}
 Let $\varphi(z_1, z_2) = (\alpha_1 z_1, \alpha_2 z_2)$, where $\alpha_1$ is a primitive $p^{th}$-root of unity and $\alpha_2$ is a primitive $q^{th}$-root of unity, $p,q \in \mathds{N}$. Let $m = l.c.m.(p,q)$. Then, 
\begin{enumerate}
  \item $\sigma(T) = \sigma_{lsf}(T) = \{\lambda \in \mathds{C}: \exists t \in \mathfrak{M}_A, \lambda^m = w_m(t)\}$.
  \item $\sigma_{a.p.}(T) = \sigma_{sf}(t) =\{\lambda \in \mathds{C}: \exists t \in \partial A, \lambda^m = w_m(t)\}$
\end{enumerate}
\end{proposition}

\begin{proof}
  The proof follows from Lemma~\ref{l1} and the fact that the set of zeros of $w_m - \lambda^m$ in $\mathds{U}^2$ is either empty or infinite.
\end{proof}

\begin{proposition} \label{p1}
  Let $\varphi(z_1, z_2) = (\alpha_1 z_1, \alpha_2 z_2)$, where $\alpha_1, \alpha_2 \in \mathds{T}$, $\alpha_1$ is a primitive $p^{th}$-root of unity and $\alpha_2$ is not a root of unity. Let $w \in \mathds{A}^2$ and let
  $(Tf)(z_1, z_2) = w(z_1, z_2)f(\varphi(z_1, z_2)), (z_1, z_2) \in \mathds{D}^2, f \in \mathds{A}^2$. We assume that $w(z_1, z_2) = z_1^s z_2^t \tilde{w}(z_1, z_2)$, where $s,t \geq 0$ and $|\tilde{w}(0,0)| >0$
  
    \begin{enumerate}
    \item If $w$ is either an invertible element of $\mathds{A}^2$ or $w$ is not invertible in $C(\mathds{T}^2)$, then 
    
    \begin{equation} \label{eq21}
      \sigma(T) = \sigma_{sf}(T) = \{ \lambda \in \mathds{C} : \exists z_1 \in \mathds{T}, |\lambda|^p = \tilde{w}_p|(z_1,0)|\}.
    \end{equation}
    \item If $w$ is invertible in $C(\mathds{T}^2)$ but not invertible in $\mathds{A}^2$, then
    
    \begin{equation}\label{eq2}
            \begin{split}
        & \sigma_{a.p.}(T) = \sigma_{sf}(T) = \{ \lambda \in \mathds{C} : \exists z_1 \in \mathds{T}, |\lambda| = |\tilde{w}_p(z_1,0)|\}, \\
        & \sigma(T) = \sigma_{lsf}(T) = \rho(T)\mathds{D} = \max \limits_{z_1 \in \mathds{T}} |\tilde{w}_p(z_1,0)|^{1/p}\mathds{D}. \\
      \end{split}
    \end{equation}
  \end{enumerate}
\end{proposition}

\begin{proof}
  Let us first assume that $p=1$, i.e. that $\varphi(z_1, z_2) = (z_1, \alpha z_2)$, where $\alpha$ is not a root of unity. It is immediate to see that if $\mu \in C(\mathds{T}^2)^\prime$ is an extreme point of the set of all $\varphi$-invariant probability measures, then there is a $z_1 \in \mathds{T}$ such that $\mu$ is the normalised Lebesgue measure on the circle $\{(z_1,z) : z \in \mathds{T}\}$, and therefore
  
  \begin{equation}\label{eq3}
    \exp \int \ln{|w|} d \mu = |\tilde{w}(z_1, 0)|.
  \end{equation}
  Therefore, in the case $p=1$ the statements~(\ref{eq21}) and~(\ref{eq2}) follow from~\cite[Corollary 2.7, Theorem 2.9]{KO} and from the obvious observation that in case (2) $def (T) = \infty$. 
  
  To finish the proof in general case we consider the operator $T^p$ and apply the spectral mapping theorem for Fredholm and semi-Fredholm spectra (see~\cite[Corollary 3.61, p.149]{Ai})
\end{proof}

\begin{proposition} \label{p4}
  Let $\varphi(z_1, z_2) = (\alpha_1 z_1, \alpha_2 z_2)$, where $\alpha_1, \alpha_2 \in \mathds{T}$, where both  $\alpha_1$ and $\alpha_2$ are not roots of unity.
  \begin{enumerate}
    \item If $w$ is invertible in $\mathds{A}^2$, then 
    
    \begin{equation}\label{eq7}
      \sigma(T) = \sigma_{sf}(T) = \{\lambda \in \mathds{C} : \rho_{min}(T) \leq |\lambda| < \rho(T)\}.
    \end{equation}
    \item If $w$ is not invertible in $C(\mathds{T}^2)$, then
    
    \begin{equation}\label{eq8}
      \sigma(T) = \sigma_{sf}(T) = \rho(T)\mathds{D}.
    \end{equation}
     \item If $w$ is invertible in $C(\mathds{T}^2)$ but not invertible in $\mathds{A}^2$, then
    \begin{equation}\label{eq9}
            \begin{split}
       & \sigma_{a.p.}(T) = \sigma_{sf}(T) = \{\lambda \in \mathds{C}: \rho_{min}(T,C(\mathds{T})) \leq |\lambda| \leq \rho(T)\}, \\
       & \sigma(T) = \sigma_{lsf}(T) = \rho(T)\mathds{D}.
      \end{split}
    \end{equation}
  \end{enumerate}
\end{proposition}

\begin{proof}
  The proof follows from Lemma~\ref{l2} and the fact that the set of zeros of $w$ in $\mathds{U}^2$ is infinite.
\end{proof}

By imposing additional conditions on the map $\varphi$  we can improve the statement of Proposition~\ref{p4}.

\begin{proposition} \label{p2}
  Let $\varphi(z_1, z_2) = (\alpha_1 z_1, \alpha_2 z_2)$, where $\alpha_1, \alpha_2 \in \mathds{T}$. Assume that $w(z_1, z_2) = z_1^s z_2^t \tilde{w}(z_1, z_2)$, where $s,t \geq 0$ and $|\tilde{w}(0,0)| >0$.  Assume additionally that either
  
 \noindent (A) $\alpha_1^p, \alpha_2^q \neq 1$ for any $p,q \in \mathds{Z}$, or
 
 \noindent (B) $\alpha_1^p =\alpha_2^q$  for some $p,q \in \mathds{N}$.
  
  \begin{enumerate}
    \item If $w$ is invertible in $\mathds{A}^2$, then 
    
    \begin{equation}\label{eq4}
      \sigma(T) = \sigma_{sf}(T) = w(0,0)\mathds{T}.
    \end{equation}
    \item If $w$ is not invertible in $C(\mathds{T}^2)$, then
    
    \begin{equation}\label{eq5}
       \sigma(T) = \sigma_{sf}(T) = \tilde{w}(0,0)\mathds{D}.
    \end{equation}
    \item If $w$ is invertible in $C(\mathds{T}^2)$ but not invertible in $\mathds{A}^2$, then
    
    \begin{equation}\label{eq6}
            \begin{split}
       & \sigma_{a.p.}(T) = \sigma_{sf}(T) = \tilde{w}(0,0)\mathds{T}, \\
       & \sigma(T) = \sigma_{lsf}(T) = \tilde{w}(0,0)\mathds{D}.
      \end{split}
    \end{equation}
  \end{enumerate}
\end{proposition}
 
\begin{proof}
 (A) The condition (A) of the proposition guarantees that the only $\varphi$-invariant probability measure on $\mathds{T}^2$ is the normalised Lebesgue measure. Hence $\rho_{min}(T) = \rho(T) = |\tilde{w}(0,0)|$.
 
 (B) Consider the operator $T^q$. Notice that 
    
    \begin{equation}\label{eq22} 
      (T^qf)(z_1,z_2) = w_q(z_1,z_2)f(\alpha_1^q z_1, \alpha_1^p z_2).
    \end{equation}
    Let $\mu$ be an extreme point of the set of all $\varphi^q$-invariant probability measures on $\mathds{T}^2$. It is easy to see that there is $(z_1, z_2) \in \mathds{T}^2$ such that 
    $supp(\mu) = cl\{(\alpha_1^{qn}z_1, \alpha_1^{pn} z_2), n \in \mathds{N}\}$. 
    
    Therefore $\exp \int \log |w_q| d\mu = |\tilde{w}_q(0,0)|$, and the statement of the theorem for the operator $T^q$ follows. It remains to apply the spectral mapping theorem for essential spectra. 
\end{proof}

As Example~\ref{e1} shows, the statement of Proposition~\ref{p2} is in general not valid in the case when
there are $p,q \in \mathds{Z}$ such that  $\alpha_1^p \alpha_2^q =1.$

\begin{example} \label{e1} Let $\alpha \in \mathds{T}$ be not a root of unity. Let 
$\varphi(z_1, z_2) = (\alpha z_1, \bar{\alpha} z_2), (z_1, z_2) \in \mathds{T}^2$, and let
\begin{equation*}
  (Tf)(z_1, z_2) = (2+z_1z_2)f(\varphi(z_1,z_2)), f \in \mathds{A}^2.
\end{equation*}
It is not difficult to see that 
\begin{equation*}
  \sigma_{sf}(T) = \sigma(T) = \{\lambda \in \mathds{C} : 1 \leq |\lambda| \leq 3\}.
\end{equation*}
\end{example}

\begin{proposition} \label{p5}
  Let $\varphi(z_1, z_2) = (\alpha z_1, \psi(z_2))$, where $\alpha$ is a primitive $p^{th}$ root of unity and $\psi$ is a parabolic Möbius transformation of $\mathds{U}$. Let $\zeta$ be the fixed point of $\psi$. Let
$r = \min \limits_{\varsigma \in \mathds{T}} |w_p(\varsigma, \zeta)|^{1/p}$ and $R = \max \limits_{\varsigma \in \mathds{T}} |w_p(\varsigma, \zeta)|^{1/p}$.
  \begin{enumerate}
    \item If $w$ is invertible in $\mathds{A}^2$, then
    
    \begin{equation}\label{eq10}
      \sigma(T) = \sigma_{sf}(T) = \mathds{AN}(r,R).
    \end{equation}
\item If $w$ is not invertible in $C(\mathds{T}^2)$, then
        \begin{equation}\label{eq12}
      \sigma(T) = \sigma_{sf}(T) = R\mathds{D}.
    \end{equation}
    \item If $w$ is not invertible in $\mathds{A}^2$ but invertible in $C(\mathds{T}^2)$, then
        \begin{equation}\label{eq 11}
      \begin{split}
      & \sigma_{a.p.}(T) = \sigma_{sf}(T) = \mathds{AN}(r, R), \\
      & \sigma(T) = \sigma_{lsf}(T) = R\mathds{D}. \\   
      \end{split}
    \end{equation}
      \end{enumerate}
  \end{proposition}
  
\begin{proof}
  (1) and (2) follow from Lemma\ref{l3}, Lemma~\ref{l5}, and~\cite[Theorem 3.2]{Ki1}. To prove (3) it suffices to notice that
$\sigma_r(T) = r\mathds{U}$ and that $def \, T = \infty$ (recall that the set of zeros of $w$ in $\mathds{U}^2$ is infinite).
\end{proof}
  
  \begin{proposition} \label{p6}
    Let $\varphi(z_1, z_2) = (\alpha z_1, \psi(z_2))$, where $\alpha$ is not a root of unity and $\psi$ is a parabolic Möbius transformation of $\mathds{U}$. Let $\zeta$ be the fixed point of $\psi$. Assume that 
    $w(z_1,z_2) = z_1^p z_2^q \tilde{w}(z_1,z_2)$, where $\tilde{w}(0,0) \neq 0$
   \begin{enumerate}
    \item If $w$ is invertible in $\mathds{A}^2$, then
    
    \begin{equation}\label{eq13}
      \sigma(T) = \sigma_{sf}(T) = w(0, \zeta)\mathds{T}.
    \end{equation}
 \item If $w$ is not invertible in $C(\mathds{T}^2)$, then
        \begin{equation}\label{eq15}
     \sigma(T) = \sigma_{sf}(T) =  \left\{
        \begin{array}{ll}
        \tilde{w}(0, \zeta)\mathds{D}  , & \hbox{if 
        $\int \limits_0^{2\pi} ln|w(e^{i\theta}, \zeta)| d\theta > -\infty$ ;} \\
          \{0\}, & \hbox{if $\int \limits_0^{2\pi} ln|w(e^{i\theta},\zeta)| d\theta = -\infty$ .}
        \end{array}
      \right.
    \end{equation}
    \item If $w$ is not invertible in $\mathds{A}^2$ but invertible in $C(\mathds{T}^2)$, then
        \begin{equation}\label{eq 14}
      \begin{split}
      & \sigma_{a.p.}(T) = \sigma_{sf}(T) =  \tilde{w}(0, \zeta)\mathds{T}, \\
      & \sigma(T) = \sigma_{lsf}(T) = \tilde{w}(0, \zeta)\mathds{D}.  \\   
      \end{split}
    \end{equation}
         \end{enumerate} 
  \end{proposition}
 
\begin{proof}
  Similar to the proof of Proposition~\ref{p5}.
\end{proof}
  
  \begin{proposition} \label{p7}
    Let $\varphi(z_1, z_2) = (\phi(z_1), \psi(z_2))$, where $\phi$ and  $\psi$ are parabolic Möbius transformations of $\mathds{U}$. Let $\varsigma$ and $\zeta$ be the fixed points of $\phi$ and $\psi$, respectively. 
   \begin{enumerate}
    \item If $w$ is invertible in $\mathds{A}^2$, then
    
    \begin{equation}\label{eq16}
      \sigma(T) = \sigma_{sf}(T) = w(\varsigma, \zeta)\mathds{T}.
    \end{equation}
 \item If $w$ is not invertible in $C(\mathds{T}^2)$, then
        \begin{equation}\label{eq18}
      \sigma(T) = \sigma_{sf}(T) = w(\varsigma, \zeta)\mathds{D}.
    \end{equation}
    \item If $w$ is not invertible in $\mathds{A}^2$ but invertible in $C(\mathds{T}^2)$, then
        \begin{equation}\label{eq 17}
      \begin{split}
      & \sigma_{a.p.}(T) = \sigma_{sf}(T) =  w(\varsigma, \zeta)\mathds{T}, \\
      & \sigma(T) = \sigma_{lsf}(T) = w(\varsigma, \zeta)\mathds{D}.  \\   
      \end{split}
    \end{equation}
     \end{enumerate} 
  \end{proposition}
\begin{proof}
  Similar to the proof of Proposition~\ref{p5}.
\end{proof}

\begin{proposition} \label{p8}
  Let $\varphi(z_1, z_2) = (\alpha z_1, \psi(z_2))$, where $\alpha$ is a $p^{th}$ primitive root of unity and $\psi$ is a hyperbolic Möbius transformation of $\mathds{U}$. Let $\zeta_1, \zeta_2$ be the fixed points of $\psi$. Assume that $|\psi^\prime(\zeta_1)| < 1$.

\noindent For any $\varsigma \in \mathds{T}$ let $a(\varsigma) = |w_p(\varsigma, \zeta_1)|^{1/p}$ and 
$b(\varsigma) = |w_p(\varsigma, \zeta_2)|^{1/p}$. 
Let $r = \min \limits_{\varsigma \in \mathds{T}} \{|a(\varsigma)|, |b(\varsigma)|\}$ and
$R = \max \limits_{\varsigma \in \mathds{T}} \{|a(\varsigma)|, |b(\varsigma)|\}$.

Also, let $\mathds{E} = \{\varsigma \in \mathrm{T} : |a(\varsigma)| \geq |b(\varsigma)|\}$.

  \begin{enumerate}
    \item If $w$ is invertible in $\mathds{A}^2$, then
    
    \begin{equation}\label{eq19}
      \begin{split}
      & \sigma(T) = \mathds{AN}(r, R).  \\
      & \sigma_{usf}(T) = \bigcup \limits_{\varsigma \in \mathds{E}} \mathds{AN}(|b(\varsigma)|, |a(\varsigma)|). \\
      & \sigma_{lsf}(T) \supseteq cl \bigcup \limits_{\varsigma \in \mathds{T} \setminus \mathds{E}} \mathds{AN}(|a(\varsigma)|, |b(\varsigma)|). \\
      \end{split}
    \end{equation}
    \item If $w$ is not invertible in $\mathds{A}^2$ but invertible in $C(\mathds{T}^2)$, then 
    \begin{equation}\label{eq20}
       \begin{split}
      & \sigma(T) = R\mathds{D}.  \\
      & \sigma_{usf}(T) = \bigcup \limits_{\varsigma \in \mathds{E}} \mathds{AN}(|b(\varsigma)|, |a(\varsigma)|). \\
      & \sigma_{lsf}(T) \supseteq r\mathds{D} \cup cl \bigcup \limits_{\varsigma \in \mathds{T} \setminus \mathds{E}} \mathds{AN}(|a(\varsigma)|, |b(\varsigma)|). \\
      \end{split}
    \end{equation}
    
    \item If $w$ is not invertible in $C(\mathds{T}^2)$, then
    \begin{equation}\label{eq23}
      \begin{split}
          & \sigma(T) = R\mathds{D}. \\
      & \sigma_{usf}(T) = \bigcup \limits_{\varsigma \in \mathds{E}} \mathds{AN}(|b(\varsigma)|, |a(\varsigma)|) \cup \bigcup \limits_{\varsigma \in \mathds{F}} |a(\varsigma)|\mathds{D}, \\
      & \text{where} \; \mathds{F} = \{\varsigma \in \mathds{T} : \exists z \in \mathds{T} \; \text{such that} \; w(\varsigma, z) = 0 \}. \\
       & \sigma_{lsf}(T) \supseteq cl \bigcup \limits_{\varsigma \in \mathds{T} \setminus \mathds{E}}
\mathds{AN}(|a(\varsigma)|, |b(\varsigma)|) \cup \bigcup \limits_{\varsigma \in \mathds{F}} |b(\varsigma)|\mathds{D}.  \\
      \end{split}
    \end{equation}
  \end{enumerate}
\end{proposition}

\begin{proof} The proof follows from Lemmas~\ref{l2}, ~\ref{l4}, and~\ref{l5}.

\end{proof}

\begin{proposition} \label{p9}
 Let $\varphi(z_1, z_2) = (\alpha z_1, \psi(z_2))$, where $\alpha$ is not a root of unity and $\psi$ is a hyperbolic Möbius transformation of $\mathds{U}$. Let $\zeta_1, \zeta_2$ be the fixed points of $\psi$. Assume that $|\psi^\prime(\zeta_1)| < 1$. Let $w(z_1, z_2) = z_1^p z_2^q \tilde{w}(z_1, z_2)$, where 
 $\tilde{w}(0,0) \neq 0$.

Let $r = \min{(|\tilde{w}(0, \zeta_1)|, |\tilde{w}(0, \zeta_2)|)}$ and $R = \max{(|\tilde{w}(0, \zeta_1)|,   |\tilde{w}(0, \zeta_2)|)}$.

\begin{enumerate}
  \item If $w$ is invertible in $\mathds{A}^2$, then 
  
  \begin{equation}\label{eq24}
  \sigma(T) = \mathds{AN}(r,R).
  \end{equation}
   
  \begin{equation}\label{eq25}
\sigma_{usf}(T) = \left\{
      \begin{array}{ll}
       \mathds{AN}(r,R), & 
        \hbox{if $|\tilde{w}(0,\zeta_2)| \leq |\tilde{w}(0,\zeta_1)|$ ;} \\
        r\mathds{T} \cup R\mathds{T}, & \hbox{if $ |\tilde{w}(0,\zeta_1)| \leq |\tilde{w}(0,\zeta_2)| $ .}
      \end{array}
    \right.
     \end{equation}

\begin{equation}\label{eq41}
   \sigma_{lsf}(T) \supseteq \left\{
      \begin{array}{ll}
       \mathds{AN}(r,R), & 
        \hbox{if $|\tilde{w}(0,\zeta_1)| \leq |\tilde{w}(0,\zeta_2)|$ ;} \\
        r\mathds{T} \cup R\mathds{T}, & \hbox{if $ |\tilde{w}(0,\zeta_2)| \leq |\tilde{w}(0,\zeta_1)| $ .}
      \end{array}
    \right.
  \end{equation}

    \item If $w$ is not invertible in $\mathds{A}^2$, but invertible in $C(\mathds{T}^2)$, then
   
   \begin{equation}\label{26}
     \sigma(T) = R\mathds{D}.
   \end{equation}
   
   \begin{equation}\label{eq27}
   \sigma_{usf}(T) = \left\{
      \begin{array}{ll}
       \mathds{AN}(r,R), & 
        \hbox{if $|\tilde{w}(0,\zeta_2)| \leq |\tilde{w}(0,\zeta_1)|$ ;} \\
        r\mathds{T} \cup R\mathds{T}, & \hbox{if $ |\tilde{w}(0,\zeta_1)| \leq |\tilde{w}(0,\zeta_2)| $ .}
      \end{array}
    \right.
  \end{equation}

\begin{equation}\label{eq42}
   \sigma_{lsf}(T) \supseteq \left\{
      \begin{array}{ll}
      R \mathds{D}, & 
        \hbox{if $|\tilde{w}(0,\zeta_1)| \leq |\tilde{w}(0,\zeta_2)|$ ;} \\
       r\mathds{D} \cup R\mathds{T} , & \hbox{if $ |\tilde{w}(0,\zeta_2)| \leq |\tilde{w}(0,\zeta_1)| $ .}
      \end{array}
    \right.
  \end{equation}

   \item  If $w$ is not invertible in $C(\mathds{T}^2)$, then
    \begin{equation}\label{eq28}
\begin{split}
   & \sigma(T) = R\mathds{D}, \\
& \sigma_{usf}(T) = |\tilde{w}(0, \zeta_1)|\mathds{D} \cup |\tilde{w}(0, \zeta_2)|\mathds{T},  \\
& \sigma_{lsf}(T) \supseteq |\tilde{w}(0, \zeta_2)|\mathds{D} \cup |\tilde{w}(0, \zeta_1)|\mathds{T}. \\
\end{split}
  \end{equation} 
 \end{enumerate}
\end{proposition}

\begin{proof} The proof follows from Lemmas~\ref{l2}, ~\ref{l4}, ~\ref{l5}, and from the results in~\cite{Ki1}. 
  
\end{proof}

\begin{proposition} \label{p10}
 Let $\varphi(z_1, z_2) = (\phi(z_1), \psi(z_2))$, where $\phi$ and $\psi$ are a parabolic and a hyperbolic Möbius transformations of $\mathds{U}$, respectively. Let $\zeta_1, \zeta_2$ be the fixed points of $\psi$. Assume that $|\psi^\prime(\zeta_1)| < 1$. Let $\varsigma$ be the fixed point of $\varphi$.

\noindent Let $r = \min{(|w(\varsigma, \zeta_1)|, |w(\varsigma, \zeta_2)|)}$ and
$R =\max{(|w(\varsigma, \zeta_1)|, |w(\varsigma, \zeta_2)|)}$.

\begin{enumerate}
  \item If $w$ is invertible in $\mathds{A}^2$, then
    \begin{equation}\label{eq29}
  \sigma(T) = \mathds{AN}(r,R).
  \end{equation}
 
    \begin{equation}\label{eq30}
   \sigma_{usf}(T) = \left\{
      \begin{array}{ll}
       \mathds{AN}(r,R), & 
        \hbox{if $|w(\varsigma,\zeta_2)| \leq |w(\varsigma,\zeta_1)|$ ;} \\
        r\mathds{T} \cup R\mathds{T}, & \hbox{if $ |w(\varsigma,\zeta_1)| \leq |w(\varsigma,\zeta_2)| $ .}
      \end{array}
    \right.
  \end{equation}

 \begin{equation}\label{eq43}
   \sigma_{lsf}(T) \supseteq \left\{
      \begin{array}{ll}
       \mathds{AN}(r,R), & 
        \hbox{if $|w(\varsigma,\zeta_1)| \leq |w(\varsigma,\zeta_2)|$ ;} \\
        r\mathds{T} \cup R\mathds{T}, & \hbox{if $ |w(\varsigma,\zeta_2)| \leq |w(\varsigma,\zeta_1)| $ .}
      \end{array}
    \right.
  \end{equation}

   \item If $w$ is not invertible in $\mathds{A}^2$, but invertible in $C(\mathds{T}^2)$, then
   
   \begin{equation}\label{31}
     \sigma(T) = R\mathds{D}.
   \end{equation}
   
   \begin{equation}\label{eq32}
   \sigma_{usf}(T) = \left\{
      \begin{array}{ll}
       \mathds{AN}(r,R), & 
        \hbox{if $|w(\varsigma,\zeta_2)| \leq |w(\varsigma,\zeta_1)|$ ;} \\
        r\mathds{T} \cup R\mathds{T}, & \hbox{if $ |w(\varsigma,\zeta_1)| \leq |w(\varsigma,\zeta_2)| $ .}
      \end{array}
    \right.
  \end{equation}

\begin{equation}\label{eq35}
  \sigma_{lsf}(T) \supseteq \left\{
      \begin{array}{ll}
       R\mathds{D}, & 
        \hbox{if $|w(\varsigma,\zeta_1)| \leq |w(\varsigma,\zeta_2)|$ ;} \\
        r\mathds{D} \cup R\mathds{T}, & \hbox{if $ |w(\varsigma,\zeta_2)| \leq |w(\varsigma,\zeta_1)| $ .}
      \end{array}
    \right.
\end{equation}
    
  \item  If $w$ is not invertible in $C(\mathds{T}^2)$, then
  
  \begin{equation}\label{eq33}
    \sigma(T) = R\mathds{D}. 
  \end{equation}
If $|w(\varsigma,\zeta_2)| \leq |w(\varsigma,\zeta_1)|$, then
\begin{equation}\label{36}
  \begin{split}
  & \sigma_{usf}(T) = R\mathds{D}, \\
& \sigma_{lsf}(T) \supseteq r\mathds{D} \cup R\mathds{T}.\\
\end{split}
\end{equation}
If $|w(\varsigma,\zeta_1)| \leq |w(\varsigma,\zeta_2)|$, then
 \begin{equation}\label{37}
  \begin{split}
  & \sigma_{lsf}(T) = R\mathds{D}, \\
& \sigma_{usf}(T) = r\mathds{D} \cup R\mathds{T}. \\
\end{split}
\end{equation}
 \end{enumerate}
\end{proposition}

\begin{proof} The proof follows from Lemmas~\ref{l2}, ~\ref{l4}, \ref{l5} and from the results in~\cite{Ki1}. 
\end{proof}
It remains to consider the case when $\varphi(z_1, z_2) = (\phi(z_1), \psi(z_2))$, where $\phi$ and $\psi$ are hyperbolic Möbius transformations of $\mathds{U}$. We denote the attracting fixed points of $\phi$ and $\psi$ as $\varsigma_1, \zeta_1$ and the repelling points as $\varsigma_2, \zeta_2$, respectively. Let $w \in \mathds{A}^2$ and let
\begin{equation}\label{eq38}
  Tf(z_1,z_2) =w(z_1,z_2)f(\phi(z_1), \psi(z_2)), f \in \mathds{A}^2, (z_1,z_2) \in \mathds{D}^2.
\end{equation}
Let $F$ be the set of all four fixed points of $\varphi$ and let $r = \min \limits_{(z_1,z_2) \in F} |w(z_1, z_2)|$, $R = \max \limits_{(z_1,z_2) \in F} |w(z_1, z_2)|$. Then,
\begin{equation}\label{39}
\sigma(T) =  \left\{
    \begin{array}{ll}
      \mathds{AN}(r,R), & \hbox{if $w$ is invertible in $\mathds{A}^2$,} \\
      R\mathds{D}, & \hbox{if  $w$ is not invertible in $\mathds{A}^2$ .}
    \end{array}
  \right.
\end{equation}
 The partial description of the essential spectra of $T$ depends on the order of the numbers 
$|w(z_1, z_2)|, (z_1, z_2) \in F$ on the real line and on invertibility or non-invertibility of the weight $w$ in $\mathds{A}^2$ or in $C(\mathds{T}^2)$, and therefore there are 72 distinct cases.
One of these cases is considered in the following proposition.

\begin{proposition} \label{p11}
 Let $T$ be the operator of the form~(\ref{eq38}). Assume that $w$ is invertible in $\mathds{A}^2$. Let $r = |w(\varsigma_1, \zeta_1)| < |w(\varsigma_1, \zeta_2)| < |w(\varsigma_2, \zeta_1)| < |w(\varsigma_2, \zeta_2)| = R$. Then,
\begin{equation}\label{eq40}
\begin{split}
  & \sigma_{lsf}(T) = \sigma(T) = \mathds{AN}(r,R), \\
& \sigma_{usf}(T) = r\mathds{T} \cup |w(\varsigma_1, \zeta_2)|\mathds{T} 
\cup |w(\varsigma_2, \zeta_1)|\mathds{T} \cup  R\mathds{T}. \\
\end{split}
\end{equation}
\end{proposition}

\begin{proof}
  The proof follows from Lemmas~\ref{l4} and~\ref{l5}.
\end{proof}


\begin{thebibliography} {99}
\bibitem{Ai} Aiena P., Fredholm and local spectral theory with applications to multipliers, Kluwer Academic Publishers (2004).
\bibitem{EE} Edmunds D.E. and Evans W.D., Spectral theory and differential operators, Clarendon Press (1987).
\bibitem{Kam} Kamowitz H., The spectra of a class of operators on the disc algebra., Indiana University Mathematics Journal, 27, No 4, 581-610 (1978). 
\bibitem{Ki} Kitover A.K., Spectrum of weighted composition operators: part 1.
Weighted composition operators on C(K) and uniform algebras. Positivity, 15, 639 - 659 (2011).
\bibitem{Ki1} Kitover A.K., Spectrum of Weighted Composition Operators. Part III: Essential Spectra of Some Disjointness Preserving Operators on Banach Lattices. Ordered Structures and Applications: Positivity VII
Trends in Mathematics, Springer International Publishing, 233–261 (2016). 
\bibitem{KO} Kitover A.K., Orhon M., Spectrum of weighted composition operators, part IX. The spectrum and essential spectra of some weighted composition operators on uniform algebras. arXiv:2307.01112v2 [math.SP], submitted to the proceedings of the conference ICMASC'2024, Porto, Portugal, 2024. 
\bibitem{Ru} Rudin W. Function theory in polydiscs, W. A. BENJAMIN, INC. (1969)
\end{thebibliography}
\end{document}